\documentclass[12pt]{article}
\usepackage{amsmath,amssymb,amsthm,color,enumerate,comment,centernot,enumitem,url,cite, mathrsfs}
\usepackage{tikz}
\usetikzlibrary{arrows}
\usepackage{graphicx,relsize}
\usepackage{mathtools, enumitem}
\usepackage{array}

\numberwithin{equation}{section}

\theoremstyle{plain}

\newtheorem{theorem}[subsection]{{Theorem}}

\newtheorem{corollary}[subsection]{{Corollary}}

\theoremstyle{definition}
\newtheorem{definition}[subsection]{{Definition}}

\theoremstyle{remark}
\newtheorem{remark}[subsection]{{Remark}}

\newtheoremstyle{RepTheorem} 
	{\topsep}{\topsep}
	{\itshape}
	{}
	{\bfseries}
	{.}
	{ }
	{\thmname{#1}\thmnote{ \bfseries #3}}
\theoremstyle{RepTheorem}
\newtheorem{reptheorem}[subsection]{Theorem}

\def\Z {{\mathbb Z}}
\def\F {{\mathbb F}}
\def\Q {{\mathbb Q}}

\def\A {{\mathbb A}}
\def\G {{\mathbb G}}
\def\P {{\mathbb P}}

\def\mcL {{\mathcal L}}
\def\mcP {{\mathcal P}}

\def \w {{\textnormal{\textbf{w}}}}
\def \u {{\textnormal{\textbf{u}}}}
\def \v {{\textnormal{\textbf{v}}}}

\newcommand{\Div}{\textnormal{Div}}

\newcommand{\Ht}{\textnormal{ht}}

\newcommand{\ord}{\textnormal{ord}}
\newcommand{\Pic}{\textnormal{Pic}}

\newcommand{\mapsfrom}{\mathrel{\reflectbox{\ensuremath{\mapsto}}}}

\newcommand{\defeq}{\stackrel{\textnormal{def}}{=}}

\begin{document}

\title{Points of Bounded Height on Weighted Projective Stacks Over Global Function Fields}

\author{Tristan Phillips}

\date{}

\maketitle

\begin{abstract}
In this note we give exact formulas (and asymptotics) for the number of rational points of bounded height on weighted projective stacks over global function fields. 
\end{abstract}

\section{Introduction}

In 1979 Schanuel proved a beautiful asymptotic formula for the number of rational points of bounded height on a projective space $\P^n$ over a number field $K$ \cite{Sch79}. In 1998 Schanuel's result was generalized to the case of weighted projective spaces by Deng \cite{Den98} (see also \cite[Theorem 3.7]{BN22}, where it is shown that Deng's proof can easily be modified to work for weighted projective stacks). Recently Darda has given a new proof of Deng's result using height zeta functions \cite{Dar21}.

\begin{theorem}[Schanuel-Deng-Darda] \label{thm:SDD}
Let $K$ be a number field of degree $d$ over $\Q$, with class number $h$, regulator $R$, discriminant $\Delta(K)$, $r_1$ real places, $r_2$ complex places, Dedekind zeta function $\zeta_K(s)$, and which contains $\omega$ roots of unity. Let $\w=(w_0,\dots,w_n)$ be an $(n+1)$-tuple of positive integers and set $|\w|\defeq w_0+w_1+\cdots+w_n$ and
$\gcd(\w,\omega)\defeq\gcd\{w_0,w_1,\dots,w_n,\omega\}$.

Then the number of $K$-rational points on the weighted projective stack $\mcP(\w)$ of height at most $B$ is
\[
\frac{h R (2^{r_1}(2\pi)^{r_2})^{n+1}|w|^{r_1+r_2-1}\gcd(\w,\omega)}{\zeta_K(|\w|)|\Delta(K)|^{(n+1)/2}\omega } B^{|\w|}+\begin{cases}
O(B\log(B)) & \text{ if } \mcP(\w)(K)=\P^1(\Q)\\
O(B^{|\w|-\min_i\{w_i\}/d}) & \text{otherwise.}
\end{cases}
\]
\end{theorem}
     
 Schanuel's Theorem, which is the case of Theorem \ref{thm:SDD} with $\w=(1,\dots,1)$, was extended to global function fields independently  by DiPippo \cite{Dip90} and Wan \cite{Wan91}. In order to state their result, we set some notation.

Let $X$ be an absolutely irreducible smooth projective curve of genus $g$ over $\F_q$. Let $\F_q(X)$ denote the function field of $X$ and let $h$ denote the class number of $\F_q(X)$.

Let $\Div(X)$ denote the set of divisors on $X$ and let 
\[
\Div^+(X)\defeq\left\{\sum_{P\in X} n_P P\in\Div(X):n_P\in \Z_{\geq 0} \text{ for all } P\right\}
\]
denote the set of effective divisors on $X$.
Define by
\[
 Z(X,t)\defeq\exp\left(\sum_{d=1}^\infty \# X(\F_{q^d}) \frac{t^d}{d}\right)=\sum_{D\in \Div^+(X)} t^{\deg(D)}
\]
the classical zeta function of $X$.
Let $\zeta_X(s)=Z(X,q^{-s})$ denote the Dedekind zeta function of $X$.

 We now define the \emph{height} of an $\F_q(X)$-rational point of projective space $x=[x_0:x_1:\cdots:x_n]\in \P^n(\F_q(X))$. For $y\in \F_q(X)$, let $(y)$
  denote the corresponding principal divisor. Let $\inf_{i}(x_i)$ denote the greatest divisor $D$ of $X$ such that $D\leq (x_i)$ for all $i$. Then we define the \emph{height} of $x$ to be
\[
\Ht_X(x)\defeq-\deg(\inf_i(x_i)).
\] 
Let
\[
A_d(\P^n)\defeq \{x\in \P^n(\F_q(X)) : \Ht_X(x)=d\}
\]
denote the number of $\F_q(X)$-rational points on $\P^n$ of height  $d$.

\begin{theorem}[DiPippo-Wan]\label{thm:DiPippo-Wan}
With notation as above, for any $\varepsilon>0$ there exist polynomials $b_i(d)$ in $d$ (whose coefficients are algebraic numbers) and algebraic integers $\alpha_i$, with $|\alpha_i|=\sqrt{q}$, such that
\begin{align*}
A_d(\P^n)&=\frac{hq^{(n+1)(1-g)}}{\zeta_X(n+1)(q-1)}q^{(n+1)d}+\sum_{i=1}^{2g} b_i(d)\alpha_i^d\\
&=\frac{hq^{(n+1)(1-g)}}{\zeta_X(n+1)(q-1)}q^{(n+1)d}+O(q^{d/2+\varepsilon}).
\end{align*}
\end{theorem}

\begin{remark}
Recently the author has given a new proof of the asymptotic in Theorem \ref{thm:DiPippo-Wan} using geometry-of-numbers techniques \cite{Phi24}, in the spirit of Schanuel's proof of Theorem \ref{thm:SDD} when $\w=(1,\dots,1)$. 
\end{remark}

In this article we extend Theorem \ref{thm:DiPippo-Wan} to the case of weighted projective stacks, giving an analog of Theorem \ref{thm:SDD} for global function fields. Let $A_d(\w)$ denote the number of $\F_q(X)$-rational points of height $d$ on the weighted projective stack $\mcP(\w)$.

\begin{theorem}\label{thm:Main}
Let notation be as above, and set $w_{\min}\defeq \min_i\{w_i\}$. Then there exist polynomials $b_i(d)$ in $d$ (whose coefficients are algebraic numbers) and algebraic integers $\alpha_i$, with $|\alpha_i|=\sqrt{q}$, such that the number of $K$-rational points on the weighted projective stack $\mcP(\w)$ of height $d$ is
\begin{align*}
A_d(\w)
&=\sum_{\substack{\u\in 2^{\w}\\ |\u|\geq 2}} \sum_{\substack{\v\in 2^\w \\ \u\subseteq \v}} \frac{\gcd(\v,q-1)(-1)^{\#\v-\#\u}hq^{\#\u(1-g)}}{(q-1)\zeta_X(|\u|)} q^{d |\u|} +\sum_{i=1}^{2g} b_i(d) \alpha_i^d\\
&=\frac{h q^{\#\w(1-g)}\gcd(\w,q-1)}{\zeta_X(|\w|)(q-1)} q^{|\w|d}+O\left(q^{d(|\w|-w_{\min})}\right).
\end{align*}
\end{theorem}

The idea of the proof is to show that the height zeta function associated to $\mcP(\w)$ over $\F_q(X)$ is rational. The methods are similar to those used by DiPippo and Wan. In \cite[Theorem 4.1.1]{Phi24}, the author gives an alternate proof of the asymptotic in Theorem \ref{thm:Main} using geometry-of-numbers techniques. The asymptotic in Theorem \ref{thm:Main} can also be obtained as a special case of the main result of \cite{DY23A}.

Counting points on weighted projective stacks fits into a more general framework of counting points of bounded height on stacks, as outlined in \cite{ESZB23} and \cite{DY24}. 

A motivation for counting points on stacks is that many interesting moduli spaces are stacks. Some key examples are the moduli spaces of curves of genus $g$; of Abelian varieties of dimension $d$; and of number fields with a fixed degree, bounded discriminant, and prescribed Galois group. Questions about counting such objects can be viewed as questions about counting points on stacks. In the context of this paper, results for counting points on weighted projective stacks can be used to count elliptic curves, hyperelliptic curves, and Drinfeld modules (see, e.g., \cite{BN22,HP19,Phi22a, Phi22c,  HP23, PS21, BPS22a, BPS22b, Phi24}).

 Just as Schanuel's theorem was a catalyst for the Batyrev-Manin Conjecture (for counting points on projective varieties) we hope that this note will motivate further work on counting points on weighted projective substacks. Darda and Yasuda have formulated a stacky analog of the Batyrev-Manin Conjecture \cite{DY24} (see also \cite[Conjecture 4.14]{ESZB23}). This generalization of the Batyrev-Manin conjecture encompasses not only the Batyrev-Manin Conjecture for varieties, but also Malle's Conjecture (for counting number fields of fixed degree, bounded discriminant, and with a prescribed Galois group).   
Beyond the body of work on the Batyrev-Manin Conjecture and Malle's conjecture, there has been recent work towards the stacky Batyrev-Manin conjecture outside of these cases (see, e.g., \cite{NX20, DY23B, Yin23, DY23A}). 

\begin{remark}
Computing exact formulas, or even asymptotics, for the number of points of bounded height on a stack can be quite difficult in general, one can nevertheless ask for upper bounds. For varieties, one of the main tools in this direction has been the \emph{determinant method}. This method was pioneered by Bombieri and Pila \cite{BP89} and further developed by Heath-Brown \cite{HB02} and Salberger \cite{Sal07,Sal23}. Most relevant to this paper are the extensions of the determinant method to global function fields \cite{Sed17,PS22} and to weighted projective spaces over the rational numbers \cite{Xia17}. It would be interesting to study the extent to which the determinant method can be extended to stacks over arbitrary global fields.
\end{remark}

\section{Weighted projective stacks}

\begin{definition}[weighted projective stack]\label{def:WPS}
Given an $(n+1)$-tuple of positive integers $\w=(w_0,w_1,\dots,w_n)\in \Z_{\geq 0}$, the \emph{weighted projective stack} $\mcP(\w)$ is defined to be the quotient stack
\[
\mcP(\w)\defeq [(\A^{n+1}-\{0\})/\G_m],
\]
with respect to the action
\begin{align*}
\ast_\w:\mathbb{G}_m\times(\A^{n+1}-\{0\}) &\to (\A^{n+1}-\{0\})\\
(\lambda,(x_0,\dots,x_n)) &\mapsto \lambda\ast_\w (x_0,\dots,x_n)\defeq (\lambda^{w_0}x_0,\dots, \lambda^{w_n}x_n).
\end{align*}
\end{definition}

Note that the weighted projective stack with weights $\w=(1,1,\dots,1)$ is the usual projective space, $\P^n$. 

\begin{remark} As a stack $\mcP(\w)$ is
 smooth (since it is the quotient stack of a smooth scheme by a smooth group scheme) and
proper (by the valuative criterion for stacks). The point $[(a_0,\dots,a_n)]\in\mcP(w_0,\dots,w_n)$ has stabilizer $\mu_m$ where $m=\gcd(w_i : a_i\neq 0)$. When $K$ is a field of characteristic relatively prime to $m$, we have that $\mu_m$ is finite and reduced over $K$. When $K$ is not relatively to $m$ then $\mu_m$ is not reduced over $K$. It follows that $\mcP(w_0,\dots,w_n)$ is a Deligne-Mumford stack over a field $K$ if and only if the characteristic of $K$ is relatively prime to each of the weights $w_i$.  In particular, over fields of characteristic $p$ the weighted projective stack $\mcP(w_0,\dots,w_n)$ is a Deligne-Mumford stack if and only if $p\nmid w_i$ for all $i$. 
\end{remark}

\begin{definition}[weighted divisor]
Let $y\in \F_q(X)$ and $w\in \Z_{> 0}$. The \emph{$w$-weighted divisor} of $y$ is defined to be
\[
(y)_w\defeq \sum_{P\in X} \left\lfloor \frac{\ord_P(y)}{w}\right\rfloor P,
\]
where $P$ ranges over the points of $X$.
\end{definition}

\begin{definition}[(logarithmic) height]
The \emph{(logarithmic) height} of an $\F_q(X)$-rational point, $y=[y_0:y_1:\dots:y_n]\in\mcP(\w)(\F_q(X))$, of a weighted projective stack is defined as
\begin{align*}
\Ht_{\w,X}(y)\defeq -\deg\left(\inf_i ((y_i)_{w_i})\right).
\end{align*}
\end{definition}

\begin{remark}
This height can also be realized as the \emph{stacky height} associated to the tatological bundle of $\mcP(\w)$ (see \cite{ESZB23} for details).
\end{remark}

We shall often abbreviate $\Ht_{\w,X}$ to $\Ht_{\w}$ when the function field is clear from context.

Define a counting function
\[
A_d(\w)\defeq \#\{y\in \mcP(\w)(\F_q(X)): \Ht_{\w,X}(y)=d\}
\]
which counts the number of $\F_q(X)$-rational points of height $d$ on $\mcP(\w)$. For example
\[
A_d(\underbrace{1,1,\dots,1}_{n+1})=A_d(\P^n).
\]

We now define weighted versions of linear equivalence, Picard groups, and Riemann-Roch spaces.

\begin{definition}[$w$-linearly equivalent divisors]
 Let $w\in\Z_{>0}$. Two divisors $D,D'\in \Div(X)$ are said to be \emph{$w$-linearly equivalent} if $D-D'=(f)_w$ for some $f\in \F_q(X)$.
\end{definition} 

\begin{definition}[$w$-Picard group]
The \emph{$w$-Picard group} is defined to be the group of $w$-equivalence classes of divisors on $X$, and is denoted $\Pic_w(X)$.
\end{definition}

\begin{definition}[$w$-Riemann-Roch space]
For $D\in \Div(X)$ we define the \emph{$w$-Riemann-Roch space} attached to $D$ as
\[
\mcL_w(D)\defeq \{f\in \F_q(X): (f)_w+D\geq 0\} \cup \{0\}.
\]
Set $\ell_w(D)\defeq \dim_{\F_q}(\mcL_w(D))$.
\end{definition} 

When $w=1$ we recover the usual definitions of Picard group and Riemann-Roch space. In these cases we will write $\Pic(X)$ for $\Pic_1(X)$, $\mcL(D)$ for $\mcL_1(D)$, and $\ell(D)$ for $\ell_1(D)$.

\section{Weighted divisor zeta functions}

Let $\w=(w_0,\dots,w_n)$ be an $(n+1)$-tuple of positive integers. Let $\#\w$ denote the length of $\w$, let $|\w|$ denote the sum of components $w_0+w_1+\cdots+w_n$, and let $2^{\w}$ denote the set of sub-tuples of $\w$. Denote by $\gcd(\w,q-1)$ the greatest common divisor of the set $\{w_0,w_1,\dots,w_n, q-1\}$.

 Consider the $\w$-weighted action of $\F_q^\times$ on $\F_q(X)^{\#\w}-\{0\}$ induced by $\ast_\w$ (as defined in Definition \ref{def:WPS}). With respect to this action, define the following set of $\F_q^\times$-orbits of points in the affine cone of $\mcP(\w)$:
\[
B_\w(D)\defeq \left\{ x=(x_0,\dots,x_n)\in \left(\F_q(X)^{\#\w}-\{0\}\right)/\F_q^\times : -\inf_{i}(x_i)_{w_i}\leq D\right\}.
\]

For what follows we will choose a degree one divisor $D^\ast$ on $X$ once and for all (or at least for the remainder of this paper).

\begin{definition}[weighted divisor zeta function]
Let $D_1,\dots,D_h$ generate the torsion part of the divisor class group. For $\w=(w_0,\dots,w_n)$ define the \emph{$\w$-weighted divisor zeta function} of $X$ to be
\[
Z_\w(X,t)\defeq \sum_{j=1}^h\sum_{d=0}^\infty  \# B_{\w}(D_j+dD^\ast) t^{d}.
\]
\end{definition}

\begin{theorem}[Rationality of weighted divisor zeta functions]\label{thm:DivZetaRationality}
For each $\u\in 2^\w$, there exists an integral polynomial $P_\u(t)$ of degree at most $1+(2g-2)/\min_i(u_i)$ such that
\[
Z_\w(X,t)=\sum_{\v\in 2^{\w}} \sum_{\u\in 2^\v} \frac{\gcd(\v,q-1)(-1)^{\#\v-\#\u} P_\u(t)}{(q-1)(1-q^{|\u|}t)}.
\] 
\end{theorem}

\begin{proof}

 For $\u=(u_0,\dots,u_r)\in \Z^{r+1}_{>0}$, set
\[
Z(\u,t)\defeq \sum_{j=1}^h\sum_{d=0}^\infty \left(\prod_{i=0}^{\#\u-1} q^{\ell(u_i(D_j+dD^\ast))}\right)t^d.
\]
Riemann-Roch implies that for any divisor $D$ of degree one, if $u_i d > 2g-2$ then $\ell(u_idD)=u_id-g+1$. Therefore there are integral polynomials $Q_\u(t)$ and $P_\u(t)$, of degrees at most $(2g-2)/\min_{i}(u_i)$ and $1+(2g-2)/\min_{i}(u_i)$ respectively, such that
\begin{align*}
Z(\u,t)&=Q_\u(t) + \sum_{d=0}^\infty h\prod_{j=0}^{\#\u-1} q^{u_jd-g+1}t^d\\
&= Q_\u(t)+ q^{\#\u(1-g)} h\sum_{d=0}^\infty q^{|\u|d}t^d\\
&=Q_\u(t)+\frac{h q^{\#\u\cdot(1-g)}}{1-q^{|\u|}t}\\
&=\frac{P_\u(t)}{1-q^{|\u|}t}.
\end{align*}

Let $x=(x_0,\dots,x_n)\in B_\w(D)$ and let $\v\subset 2^\w$ be the subtuple of $\w$ corresponding to the non-zero coordinates of $x$ (e.g., if $\w=(1,2,3)$ and $x=(2,0,7)$, then $\v=(1,3)$). Then the number of elements of $\F_q(X)^{\#\w}-\{0\}$ contained in $x$ is $\gcd(\v, q-1)/(q-1)$. It follows that
\begin{align*}
\# B_\w(D)
&= \sum_{\v\in 2^{\w}} \frac{\gcd(\v,q-1)}{q-1} \prod_{v_i\in \v} (\#\mcL_{v_i}(D)-1)\\
&=\sum_{\v\in 2^{\w}}\frac{\gcd(\v,q-1)}{q-1} \sum_{\u\in 2^\v} (-1)^{\#\v-\#\u} \prod_{u_i\in \u} q^{\ell_{u_i}(D)}.
\end{align*}
We now have that
\begin{align*}
Z_\w(X,t)
&= \sum_{j=1}^h\sum_{d=0}^\infty  \sum_{\v\in 2^{\w}}\frac{\gcd(\v,q-1)}{q-1} \sum_{\u\in 2^\v} (-1)^{\#\v-\#\u} \left(\prod_{u_i\in \u} q^{\ell_{u_i}(D_j+dD^\ast)}\right) t^d\\
&= \sum_{\v\in 2^{\w}}\frac{\gcd(\v,q-1)}{q-1} \sum_{\u\in 2^\v} (-1)^{\#\v-\#\u} Z(\u,t)\\
&=\sum_{\v\in 2^{\w}} \sum_{\u\in 2^\v} \frac{\gcd(\v,q-1)(-1)^{\#\v-\#\u} P_\u(t)}{(q-1)(1-q^{|\u|}t)}.
\end{align*}
\end{proof}

\section{Height zeta functions of weighted projective stacks}

\begin{definition}[height zeta function]
Define the \emph{height zeta function} of the weighted projective stack $\mcP(\w)$ defined over the function field $\F_q(X)$ of the curve $X$ to be
\[
Z(\mcP(\w)(\F_q(X)),t)\defeq \sum_{d=0}^\infty \# A_d(\w)t^d.
\]
\end{definition}

\begin{theorem}\label{thm:RationalityProjZeta}
The height zeta function is related to the divisor zeta function as follows,
\[
Z(\mcP(\w)(\F_q(X)),t)=\frac{Z_{\w}(X,t)}{Z(X,t)}.
\]
\end{theorem}

\begin{proof}
For a divisor $D\in\Div(X)$ define the set 
\[
N(D)\defeq \{D'\in\Div^+(X) : [D]=[D'] \text{ in }\Pic(X)\}
\]
of effective divisors linearly equivalent to $D$.

We now show that there is a bijection
\[
 \bigsqcup_{\substack{[D']\in \Div^+(X)\\ \deg(D')=d'\leq d}} N(D')\times A_{d-d'}(\w)\leftrightarrow
 \bigsqcup_{\substack{[D]\in \Pic(X)\\ \deg(D)=d}} B_\w(D). 
\]
Elements on the left can be specified by pairs $(D'',z)$, where $D''\in N(D')$ and $z\in A_{d-d'}(\w)$. Elements on the right can be specified by pairs $(D,y)$, where $D$ is a degree $d$ divisor and $y\in B_\w(D)$.  The bijection is then given by
\begin{align*}
(D'',z)&\mapsto (D''-\inf_{i}(z_i)_{w_i}, z)\\
(D+\inf_{i}(y_i)_{w_i}, y)&\mapsfrom (D,y).
\end{align*}
It follows that
\[
Z(X,t)\cdot Z(\mcP(\w)(\F_q(X)),t)= Z_\w(X,t).
\]
\end{proof}

Theoerm \ref{thm:RationalityProjZeta} together with Theorem \ref{thm:DivZetaRationality} shows that $Z(\mcP(\w)(\F_q(X)),t)$ is a rational function. From this we can derive an exact formula for $A_d(\w)$:

\begin{reptheorem}[\ref{thm:Main}]
There exist algebraic integers $\alpha_i$ with $|\alpha_i|=\sqrt{q}$ and polynomials $b_i(d)$ in $d$ (whose coefficients are algebraic numbers) such that
\[
\#A_d(\w)=\sum_{\substack{\u\in 2^{\w}\\ |\u|\geq 2}} \sum_{\substack{\v\in 2^\w \\ \u\subseteq \v}} \frac{\gcd(\v,q-1)(-1)^{\#\v-\#\u}hq^{\#\u(1-g)}}{(q-1)\zeta_X(|\u|)} q^{d |\u|} +\sum_{i=1}^{2g} b_i(d) \alpha_i^d.
\]
\end{reptheorem}

\begin{proof}
It is well known that the classical zeta function $Z(X,t)$ is a rational function and can be written as
\[
Z(X,t)=\frac{P(t)}{(1-t)(1-qt)},
\]
where $P(t)=\prod_{i=1}^{2g}(1-\alpha_i t)$ whith $\alpha_i$ algebraic integers satisfying $|\alpha_i|=\sqrt{q}$. Combining this fact with Theorem \ref{thm:DivZetaRationality} we have that
\begin{align}\label{eq:ZetaFraction}
\frac{Z_\w(X,t)}{Z(X,t)}=\sum_{\v\in 2^{\w}}\sum_{\u\in 2^\v}\frac{\gcd(\v,q-1)(-1)^{\#\v-\#\u}(1-t)(1-qt)P_\u(t)}{(q-1)(1-q^{|\u|}t)P(t)}.
\end{align}
Note that this is a rational function which has only simple poles. These poles can only possibly occur at values of $t$ in the following set,
$$\{q^{-|\u|}: \u\in 2^{\w}\text{ with }|\u|\geq 2\}.$$
Note also that the degree of the numerator in  (\ref{eq:ZetaFraction}) is at most $2g$ since $\deg(P_\u(t))\leq 2g-2$. By Theorem \ref{thm:RationalityProjZeta} and a partial fractions decomposition, we have that there are rational numbers $a_\u$ for $\u\in 2^\w$, and polynomials $b_i(d)$ for $1\leq i\leq 2g$, such that for all $d\geq 1$,
\[
\# A_d(\w)=\sum_{\substack{\u\in 2^{\w}\\ |\u|\geq 2}} a_\u q^{d |\u|} +\sum_{i=1}^{2g} b_i(d) \alpha_i^d.
\]
Noting that
\begin{align*}
a_\u = \sum_{\substack{\v\in 2^\w \\ \u\subseteq \v}} \frac{\gcd(\v,q-1)(-1)^{\#\v-\#\u}hq^{\#\u(1-g)}}{(q-1)\zeta_X(|\u|)},
\end{align*}
we obtain the desired formula.
\end{proof}

From Theorem \ref{thm:Main} we can easily derive the following asymptotic:

\begin{corollary}
For $\w=(w_0,\dots,w_n)$ set $w_{\min}\defeq \min_i\{w_i\}$. Then we have the following asymptotic in $d$: 
\[
A_d(\w)=\frac{h q^{\#\w(1-g)}\gcd(\w,q-1)}{\zeta_X(|\w|)(q-1)} q^{d|\w|}+O\left(q^{d(|\w|-w_{\min})}\right).
\]
\end{corollary}

\section{Acknowledgements}

The author would like to thank Bryden Cais and C. Douglas Haessig for their encouragement during this project, and the anonymous referee for making several corrections and suggestions that improved the exposition of this paper.

\bibliographystyle{alpha}
\bibliography{bibfile}

\end{document}